\documentclass[12pt]{amsart}

\usepackage{amsfonts}
\usepackage{amsmath}
\usepackage{amsthm}
\usepackage{amssymb}
\usepackage{breqn}

\newtheorem{theorem}{Theorem}

\newtheorem{lemma}[theorem]{Lemma}

\theoremstyle{definition}

\theoremstyle{remark}

\allowdisplaybreaks

\newcommand{\Address}{{
\bigskip
\footnotesize

Ho-Hon Leung, \textsc{Department of Mathematical Sciences, United Arab Emirates University, Al Ain, United Arab Emirates}\par\nopagebreak
\textit{E-mail address}: \texttt{hohon.leung@uaeu.ac.ae}
}}

\title{Derivations on four dimensional genetic Volterra algebra}
\author{Ho-Hon Leung}
\date{}

\begin{document}

\begin{abstract}
In this paper, we describe all derivations on four dimensional genetic Volterra algebras. We show that any local derivation is a derivation of the algebra. It is a positive answer to a conjecture made by Ganikhodzhaev, Mukhamedov, Pirnapasov and Qaralleh. 
\end{abstract}

\maketitle

\section{Introduction} \label{section1}

Let ${\bf x}=(x_1,x_2,x_3,x_4), {\bf y}=(y_1,y_2,y_3,y_4)\in\mathbb{R}^4$. We introduce a multiplication rule on $\mathbb{R}^4$ by 
\begin{align}
({\bf x}\ast{\bf y})_k&=\sum_{i,j=1}^4 p_{ij,k}x_i y_j
\end{align}for $k=1,2,3,4$. The algebraic properties of $(\mathbb{R}^4,\ast)$ depend on the structure constants $p_{ij,k}$'s for all $i,j,k=1,2,3,4$. $(\mathbb{R}^4,\ast)$ is called {\it an 4-dimensional genetic Volterra algebra} if the structure constants satisfy the following properties: 
\[ p_{ij,k}\geq 0, \mbox{  }p_{ij,k}=p_{ji,k},\mbox{  }\sum_{k=1}^4 p_{ij,k}=1,\mbox{   }i,j,k\in\{1,2,3,4\},\]and \[p_{ij,k}=0 \mbox{   if  }k\notin\{i,j\},\mbox{   for all  }i,j,k\in\{1,2,3,4\}.\]It is obvious that 
\begin{align}
\label{equation1} p_{ii,i}=1,\mbox{  and   }p_{ij,i}+p_{ij,j}=1,\mbox{   for all  }i,j\in\{1,2,3,4\}, i\neq j.
\end{align}Genetic Volterra algebra has its origin in the study of population genetics. For examples, the use of abstract algebra in the study of genetics can be found in the work of Etherington \cite{Etherington}, Reed \cite{Reed} and Worz-Busekros \cite{Busekros}. 

A {\it derivation} on an algebra $(A,\ast)$ is a linear map $D\colon A\rightarrow A$ such that \[D(u\ast v)=D(u)\ast v + u\ast D(v)\]for all $u,v\in A$. If $D=0$, then it is called a {\it trivial derivation}. The study of derivations of general genetic algebras and their interpretations can be found in the work of Gonshor \cite{Gonshor} and Holgate \cite{Holgate}.

Associativity and derivations of genetic Volterra algebras are studied in a recent paper by Ganikhodzhaev, Mukhamedov, Pirnapasov and Qaralleh (GMPQ) \cite{Farrukh}. In particular, it is shown that any local derivation is a derivation for three dimensional genetic Volterra algebra. It is natural to ask whether such result would be true for any $n$-dimensional genetic Volterra algebra (Conjecture 5.6 in GMPQ's paper \cite{Farrukh}). In this article, we describe all derivations of four dimensional genetic Volterra algebra in Section \ref{section2}. As a corollary, we show that the conjecture is true for $n=4$ in Section \ref{section3}. To simplify our notations, we denote the four dimensional genetic Volterra algebra by $(\mathbb{R}^4,\ast)$. 

\section{Main results} \label{section2}

Let $\{{\bf e_1, e_2, e_3, e_4}\}$ be the standard basis of $\mathbb{R}^4$ and $[4]=\{1,2,3,4\}$. By (\ref{equation1}), 
\begin{align}
\label{equation2} {\bf e_i}\ast {\bf e_i}={\bf e_i}, \mbox{  }{\bf e_i}\ast{\bf e_j}=p_{ij,i}{\bf e_i}+p_{ij,j}{\bf e_j}, \mbox{  }i,j\in [4].
\end{align}We note that $(\mathbb{R}^4,\ast)$ is a commutative algebra.

Let $D\colon\mathbb{R}^4\rightarrow\mathbb{R}^4$ be a derivation on $(\mathbb{R}^4,\ast)$. As a linear transformation, $D$ can be represented by a 4x4 matrix $(d_{ij})_{i,j\in [4]}$ such that \[D({\bf e_i})=\sum_{j=1}^4 d_{ij} {\bf e_j}, \mbox{   }i\in [4].\]Let \[[4]_i= \{j\in [4]\mbox{  }|\mbox{  }p_{ij,i}=1/2\}.\]The following two lemmas will be used throughout the article.

\begin{lemma} \label{lemma1}
Let $i\in [4]$. If $j\in [4]$, $i\neq j$ and $j\notin [4]_i$, then $d_{ij}=d_{ji}=0$.
\end{lemma}

\begin{proof}
It is a special case of Lemma 5.1 in GMPQ's paper \cite{Farrukh}.
\end{proof}

\begin{lemma}  \label{lemma2}
Let $i,j\in [4]$. If $p_{ij,j}\neq 1/2$ for all $i,j$, then any derivation on $(\mathbb{R}^4, \ast)$ is trivial.
\end{lemma}

\begin{proof}
It is a special case of Corollary 5.2 in GMPQ's paper \cite{Farrukh}.
\end{proof}

We need to describe all non-trivial derivations of $(\mathbb{R}^4,\ast)$. Without loss of generality (up to a permutation of elements in the standard basis of $\mathbb{R}^4$), by Lemma \ref{lemma2}, we need to analyse the following cases:
 \renewcommand{\labelenumi}{\Alph{enumi}}
\begin{enumerate}
  \item $|[4]_1|=1, |[4]_2|=1, |[4]_3|=0, |[4]_4|=0$.
  \item $|[4]_1|=1, |[4]_2|=1, |[4]_3|=1, |[4]_4|=1$.
  \item $|[4]_1|=2, |[4]_2|=1, |[4]_3|=1, |[4]_4|=0$.
  \item $|[4]_1|=2, |[4]_2|=2, |[4]_3|=2, |[4]_4|=0$.
  \item $|[4]_1|=3, |[4]_2|=1, |[4]_3|=1, |[4]_4|=1$.
  \item $|[4]_1|=2, |[4]_2|=1, |[4]_3|=2, |[4]_4|=1$.
  \item $|[4]_1|=3, |[4]_2|=2, |[4]_3|=2, |[4]_4|=1$.
  \item $|[4]_1|=2, |[4]_2|=2, |[4]_3|=2, |[4]_4|=2$.
  \item $|[4]_1|=3, |[4]_2|=2, |[4]_3|=3, |[4]_4|=2$.
  \item $|[4]_1|=3, |[4]_2|=3, |[4]_3|=3, |[4]_4|=3$.
\end{enumerate}

We need another lemma to facilitate our computations in the next section.

\begin{lemma} \label{lemma2.5}
Let $(\mathbb{R}^4,\ast)$ be a four dimensional genetic Volterra algebra. The sum of each row in any derivation $D=(d_{ij})_{i,j\in [4]}$ is zero. That is, $\sum_{j=1}^4 d_{ij}=0$ for each $i\in [4]$.
\end{lemma}

\begin{proof}
Without loss of generality, we will show that $\sum_{j=1}^4 d_{1j}=0$.
\begin{align}
\nonumber D({\bf e_1}) &= D({\bf e_1}\ast {\bf e_1})={\bf e_1}\ast D({\bf e_1})+D({\bf e_1})\ast {\bf e_1}=2{\bf e_1}\ast D({\bf e_1})\\
\nonumber  \sum_{j=1}^4 d_{1j}{\bf e_j}  &= 2{\bf e_1}\ast (\sum_{j=1}^4 d_{1j} {\bf e_j})=2 \sum_{j=1}^4 d_{1j} ({\bf e_1}\ast {\bf e_j})
\end{align}By Lemma \ref{lemma1}, $d_{1j}=0$ if $j\notin [4]_1$, $j\neq i$, $p_{1j,1}\neq 1/2$. Hence,
\begin{align}
\nonumber \sum_{j=1}^4 d_{1j} {\bf e_j} &= 2d_{11}{\bf e_1} +2 \sum_{j\notin [4]_1, j\neq i} d_{1j}(\frac{1}{2}{\bf e_1}+\frac{1}{2}{\bf e_j}).
\end{align}By comparing coefficients of ${\bf e_1}$ on both sides, we get the desired result.
\end{proof}

Lemma \ref{lemma2.5} can easily be generalized to $n$-dimensional genetic Volterra algebras. 

\subsection{Case A} \label{subsection1} By Lemma \ref{lemma1} and (\ref{equation2}), we have the following results: \[ d_{13}=d_{14}=d_{23}=d_{24}=d_{34}=d_{31}=d_{32}=d_{41}=d_{42}=d_{43}=0,\]\[{\bf e_1}\ast{\bf e_2}=\frac{1}{2}{\bf e_1}+\frac{1}{2}{\bf e_2}.\]Let $a=d_{11}$ and $b=d_{21}$. By Lemma \ref{lemma2.5}, $d_{33}=d_{44}=0$. Hence, 
\[
D=(d_{ij})_{i,j\in [4]}=\left( {\begin{array}{cccc}
a & -a & 0 & 0\\
b & -b & 0 & 0\\
0 & 0 & 0 & 0 \\
0& 0 & 0 & 0\\
\end{array}}\right).
\]We compute $D({\bf e_1}\ast {\bf e_3})$ as follows:
\begin{align}
\nonumber D({\bf e_1}\ast {\bf e_3}) &= {\bf e_1}\ast D({\bf e_3})+{\bf e_3}\ast D({\bf e_1}) \\
\nonumber D(p_{13,1} {\bf e_1}+ p_{13,3}{\bf e_3})&={\bf e_3}\ast (a {\bf e_1}-a {\bf e_1}) \\
\nonumber p_{13,1}(a {\bf e_1} -a {\bf e_2}) &= a (p_{13,1} {\bf e_1}+p_{13,3} {\bf e_3})-a (p_{23,2} {\bf e_2}+p_{23,3} {\bf e_3}).
\end{align}By comparing coefficients of ${\bf e_2}$ and ${\bf e_3}$ on both sides, we get 
\begin{align}
\label{equation3} a(p_{13,1}-p_{23,2})&= 0\\
\label{equation4} a(p_{13,3}-p_{23,3})&= 0.
\end{align}We do a similar computation on $D({\bf e_2}\ast {\bf e_3})$ to get 
\begin{align}
\label{equation5} b(p_{13,1}-p_{23,2}) &= 0\\
\label{equation6} b(p_{13,3}-p_{23,3}) &= 0.
\end{align}One of $a,b$ is non-zero otherwise $D$ would be a trivial derivation. By (\ref{equation3}), (\ref{equation4}), (\ref{equation5}), (\ref{equation6}), $p_{13,1}=p_{23,2},\mbox{  } p_{13,3}=p_{23,3}$. We do the same computations on $D({\bf e_1}\ast {\bf e_4})$ and $D({\bf e_2}\ast {\bf e_4})$ to get $p_{14,1}=p_{24,2}, \mbox{  }p_{14,4}=p_{24,4}$. We summarize our results: 

\begin{theorem} \label{caseA}
Let $(\mathbb{R}^4,\ast)$ be a four dimensional genetic Volterra algebra. In Case A: $|[4]_1|=1, |[4]_2|=1, |[4]_3|=0, |[4]_4|=0$, derivations are in the form 
\[
D=(d_{ij})_{i,j\in [4]}=\left( {\begin{array}{cccc}
a & -a & 0 & 0\\
b & -b & 0 & 0\\
0 & 0 & 0 & 0 \\
0& 0 & 0 & 0\\
\end{array}}\right).
\]If $D$ is non-trivial (one of $a,b$ is non-zero), then $p_{13,1}=p_{23,2}$, $p_{13,3}=p_{23,3}$, $p_{14,1}=p_{24,2}$, $ p_{14,4}=p_{24,4}$.
\end{theorem}

\subsection{Case B} \label{subsection2} By Lemma \ref{lemma1} and (\ref{equation2}), we have the following results: \[d_{13}=d_{31}=d_{14}=d_{41}=d_{23}=d_{32}=d_{24}=d_{42}=0,\]\[{\bf e_1}\ast {\bf e_2}=\frac{1}{2} {\bf e_1}+\frac{1}{2} {\bf e_2},\mbox{  }{\bf e_3}\ast {\bf e_4}=\frac{1}{2}{\bf e_3}+\frac{1}{2}{\bf e_4}.\]Let $a=d_{11}$, $b=d_{21}$, $c=d_{33}$ and $d=d_{43}$. By Lemma \ref{lemma2.5}, 
\[
D=(d_{ij})_{i,j\in [4]}=\left( {\begin{array}{cccc}
a & -a & 0 & 0\\
b & -b & 0 & 0\\
0 & 0 & c & -c \\
0& 0 & d & -d\\
\end{array}}\right).
\]For $D({\bf e_1}\ast {\bf e_3})$, we compute it in two ways. First,
\begin{align}
\nonumber D({\bf e_1}\ast {\bf e_3})&= D(p_{13,1}{\bf e_1}+p_{13,3} {\bf e_3}) \\
\label{equation6.1}  &= p_{13,1} (a{\bf e_1}-a{\bf e_2}) +p_{13,3} (c{\bf e_3}-c{\bf e_4}).
\end{align}Alternatively, 
\begin{align}
\nonumber D({\bf e_1}\ast {\bf e_3})&={\bf e_1}\ast D({\bf e_3}) +{\bf e_3}\ast D({\bf e_1})\\
\nonumber &= {\bf e_1}\ast (c{\bf e_3}-c{\bf e_4})+{\bf e_3}\ast (a{\bf e_1}-a {\bf e_2}) \\
 \label{equation6.2} &=\begin{aligned}[t]
            &  c(p_{13,1}{\bf e_1}+p_{13,3}{\bf e_3})-c(p_{14,1} {\bf e_1}+p_{14,4}{\bf e_4})\\
            &  +a(p_{13,1}{\bf e_1}+p_{13,3} {\bf e_3})-a(p_{23,2} {\bf e_2}+p_{23,3} {\bf e_3}).
        \end{aligned}
\end{align}By comparing coefficients of ${\bf e_1}$, ${\bf e_2}$, ${\bf e_3}$ and ${\bf e_4}$ in (\ref{equation6.1}) and (\ref{equation6.2}) respectively, 
\begin{align}
\label{equation7} c(p_{13,1}-p_{14,1}) &= 0, \\
\label{equation8} a(p_{13,1}-p_{23,2}) &= 0,\\
\label{equation9} a(p_{13,3}-p_{23,3})&=0,\\
\label{equation10} c (p_{13,3}-p_{14,4}) &= 0.
\end{align}We do the same computation for $D({\bf e_1}\ast {\bf e_4})$ to get 
\begin{align}
\label{equation11} d(p_{13,1}-p_{14,1}) &= 0,\\
\label{equation12} a(p_{14,1}-p_{24,2}) &= 0, \\
\label{equation13} d(p_{13,3}-p_{14,4}) &= 0,\\
\label{equation14} a(p_{14,4}-p_{24,4}) &= 0.
\end{align}We do the same computation for $D({\bf e_2}\ast {\bf e_3})$ to get
\begin{align}
\label{equation15} b(p_{13,1} -p_{23,2}) &= 0, \\
\label{equation16} c(p_{23,2} -p_{24,2}) &= 0, \\
\label{equation17} b(p_{13,3} -p_{23,3}) &= 0, \\
\label{equation18} c(p_{23,3} -p_{24,4}) &= 0.
\end{align}Similarly, we do the same computation for $D({\bf e_2}\ast {\bf e_4})$ to get
\begin{align}
\label{equation19} b(p_{14,1} -p_{24,2}) &= 0,\\
\label{equation20} d(p_{23,2}-p_{24,2}) &= 0,\\
\label{equation21} d(p_{23,2} -p_{24,4}) &= 0,\\
\label{equation22} b(p_{14,4}-p_{24,4}) &= 0.
\end{align} If both $a, b$ are zeros, then the derivation $D$ is reduced to the derivation appeared in Theorem \ref{caseA}. The same thing can be said for the case when both $c,d$ are zeros.

If one of $a,b$ is non-zero and one of $c,d$ is non-zero, then $p_{13,1}=p_{14,1}$, $p_{13,3}=p_{14,4}$ by (\ref{equation7}), (\ref{equation10}), (\ref{equation11}), (\ref{equation13}); $p_{13,1}=p_{23,2}$, $p_{13,3}=p_{23,3}$ by (\ref{equation8}), (\ref{equation9}), (\ref{equation15}), (\ref{equation17}); $p_{14,1}=p_{24,2}$, $p_{14,4}=p_{24,4}$ by (\ref{equation12}), (\ref{equation14}), (\ref{equation19}), (\ref{equation22}); $p_{23,2}=p_{24,2}$, $p_{23,3}=p_{24,4}$ by (\ref{equation16}), (\ref{equation18}), (\ref{equation20}), (\ref{equation21}). We summarize our results: 

\begin{theorem} \label{caseB}
Let $(\mathbb{R}^4,\ast)$ be a four dimensional genetic Volterra algebra. In Case B: $|[4]_1|=1, |[4]_2|=1, |[4]_3|=1, |[4]_4|=1$, derivations are in the form 
\[
D=(d_{ij})_{i,j\in [4]}=\left( {\begin{array}{cccc}
a & -a & 0 & 0\\
b & -b & 0 & 0\\
0 & 0 & c & -c \\
0& 0 & d & -d\\
\end{array}}\right).
\]If $D$ is non-trivial and is not the same as the derivations appeared in Theorem \ref{caseA} (that is, one of $a,b$ is non-zero and one of $c,d$ is non-zero), then $p_{13,1}=p_{14,1}=p_{23,2}=p_{24,2}$ and $p_{13,3}=p_{14,4}=p_{23,3}=p_{24,4}$.
\end{theorem}

\subsection{Case C} \label{subsection3} By Lemma \ref{lemma1} and (\ref{equation2}), we have the following results:  
\[d_{14}=d_{41}=d_{23}=d_{32}=d_{24}=d_{42}=d_{34}=d_{43}=0,\]\[{\bf e_1}\ast {\bf e_2}=\frac{1}{2} {\bf e_1}+\frac{1}{2} {\bf e_2},\mbox{  }{\bf e_1}\ast {\bf e_3}=\frac{1}{2}{\bf e_1}+\frac{1}{2}{\bf e_3}.\]Let $c=d_{11}$, $d=d_{12}$, $b=d_{21}$, $a=d_{31}$. By Lemma \ref{lemma2.5}, we have: 
\[
D=(d_{ij})_{i,j\in [4]}=\left( {\begin{array}{cccc}
c & d & -(c+d) & 0\\
b & -b & 0 & 0\\
a & 0 & -a & 0 \\
0& 0 & 0 & 0\\
\end{array}}\right).
\]For $D({\bf e_2}\ast {\bf e_3})$, we compute it in two ways. First,
\begin{align}
\nonumber D({\bf e_2}\ast {\bf e_3}) &= D(p_{23,2}{\bf e_2}+p_{23,3}{\bf e_3})=p_{23,2} D({\bf e_2})+p_{23,3} D({\bf e_3}) \\
\label{equation23}  &= p_{23,2} (b{\bf e_1}-b{\bf e_2})+p_{23,3} (a {\bf e_1}-a{\bf e_3}).
\end{align}Alternatively,
\begin{align}
\nonumber D({\bf e_2}\ast {\bf e_3}) &= {\bf e_2}\ast D({\bf e_3})+{\bf e_3}\ast D({\bf e_2}) \\
\nonumber  &= {\bf e_2}\ast (a{\bf e_1}-a{\bf e_3})+{\bf e_3}\ast (b{\bf e_1}-b{\bf e_2}) \\
   &=\begin{aligned}[t]
\label{equation24}       a (\frac{1}{2} {\bf e_1}+\frac{1}{2} {\bf e_2})-a(p_{23,2}{\bf e_2}+p_{23,3}{\bf e_3})+b(\frac{1}{2}{\bf e_1}+\frac{1}{2}{\bf e_3})\\
            -b(p_{23,2}{\bf e_2}+p_{23,3} {\bf e_3}).
        \end{aligned}
\end{align}By comparing coefficients of ${\bf e_2}$ and ${\bf e_3}$ in (\ref{equation23}) and (\ref{equation24}), we get 
\begin{align}
\label{equation25} a(\frac{1}{2} - p_{23,2})&= 0 \\
\label{equation26} b(\frac{1}{2} -p_{23,3})&= 0.
\end{align}Since $p_{23,2}\neq 1/2$ and $p_{23,3}\neq 1/2$, we get $a=0$ and $b=0$ in (\ref{equation25}), (\ref{equation26}). Hence, $D({\bf e_2})=D({\bf e_3})=0$. For $D({\bf e_1}\ast {\bf e_3})$, we compute it in two ways. First,
\begin{align}
\label{equation27} D({\bf e_1}\ast {\bf e_3}) &= \frac{1}{2} D({\bf e_1})+\frac{1}{2} D({\bf e_3})=\frac{1}{2}c{\bf e_1}+\frac{1}{2}d{\bf e_2}-\frac{1}{2}(c+d){\bf e_3}.
\end{align}Alternatively, 
\begin{align}
\nonumber D({\bf e_1}\ast {\bf e_3}) &= {\bf e_1}\ast D({\bf e_3})+{\bf e_3}\ast D({\bf e_1}) \\
\nonumber &={\bf e_3} \ast (c{\bf e_1}+d{\bf e_2}-(c+d) {\bf e_3}) \\
\label{equation28} &= c(\frac{1}{2}{\bf e_1}+\frac{1}{2}{\bf e_3})+d(p_{23,2}{\bf e_2} +p_{23,3} {\bf e_3})-(c+d){\bf e_3}.
\end{align}By comparing coefficients of ${\bf e_2}$ in (\ref{equation27}) and (\ref{equation28}), we get 
\begin{align}
\label{equation29} d(\frac{1}{2}-p_{23,2}) &= 0.
\end{align}Since $p_{23,2}\neq 1/2$, $d=0$ in (\ref{equation29}). Similarly, we do the same computation on $D({\bf e_1}\ast {\bf e_2})$, we get $c=0$. We summarize our results:

\begin{theorem} \label{caseC}
Let $(\mathbb{R}^4,\ast)$ be a four dimensional genetic Volterra algebra. In Case C: $|[4]_1|=2, |[4]_2|=1, |[4]_3|=1, |[4]_4|=0$, all derivations are trivial.
\end{theorem}

\subsection{Case D} \label{subsection4} By Lemma \ref{lemma1} and (\ref{equation2}), we have the following results: \[d_{14}=d_{24}=d_{34}=d_{41}=d_{42}=d_{43}=0,\]\[{\bf e_1}\ast {\bf e_2}=\frac{1}{2}{\bf e_1}+\frac{1}{2}{\bf e_2}, \quad {\bf e_2}\ast {\bf e_3}=\frac{1}{2}{\bf e_2}+\frac{1}{2}{\bf e_3},\quad {\bf e_1}\ast {\bf e_3}=\frac{1}{2}{\bf e_1}+\frac{1}{2}{\bf e_3}.\]By Lemma \ref{lemma2.5}, $d_{44}=0$ and hence $D({\bf e_4})=0$. Let $a=d_{11}$, $b=d_{12}$, $c=d_{21}$, $d=d_{22}$, $e=d_{31}$, $f=d_{32}$. By Lemma \ref{lemma2.5}, we summarize our results:

\begin{theorem} \label{caseD}
Let $(\mathbb{R}^4,\ast)$ be a four dimensional genetic Volterra algebra. In Case D: $|[4]_1|=2, |[4]_2|=2, |[4]_3|=2, |[4]_4|=0$, derivations are in the form 
\[
D=(d_{ij})_{i,j\in [4]}=\left( {\begin{array}{cccc}
a & b & -(a+b) & 0\\
c & d & -(c+d) & 0\\
e & f & -(e+f) & 0 \\
0& 0 & 0 & 0\\
\end{array}}\right).
\]
\end{theorem}

\subsection{Case E} \label{subsection5} By Lemma \ref{lemma1} and (\ref{equation2}), we have the following results: \[d_{23}=d_{32}=d_{24}=d_{42}=d_{34}=d_{43}=0,\]\[{\bf e_1}\ast {\bf e_2}=\frac{1}{2}{\bf e_1}+\frac{1}{2}{\bf e_2}, \quad {\bf e_1}\ast {\bf e_3}=\frac{1}{2}{\bf e_1}+\frac{1}{2}{\bf e_3},\quad {\bf e_1}\ast {\bf e_4}=\frac{1}{2}{\bf e_1}+\frac{1}{2}{\bf e_4}.\]Let $d=d_{11}$, $e=d_{12}$, $f=d_{13}$, $a=d_{21}$, $b=d_{31}$, $c=d_{41}$. By Lemma \ref{lemma2.5},
\[
D=(d_{ij})_{i,j\in [4]}=\left( {\begin{array}{cccc}
d & e & f & -(d+e+f)\\
a & -a & 0 & 0\\
b & 0 & -b & 0 \\
c & 0 & 0 & -c\\
\end{array}}\right).
\]We compute $D({\bf e_2}\ast {\bf e_3})$ in two ways. First, 
\begin{align}
\nonumber D({\bf e_2}\ast {\bf e_3}) &= p_{23,2} D({\bf e_2}) +p_{23,3} D({\bf e_3}) \\
\label{equation29.3}  &= p_{23,2}(a{\bf e_1}-a{\bf e_2})+p_{23,3}(b{\bf e_1}-b{\bf e_3}). 
\end{align}Alternatively,
\begin{align}
\nonumber D({\bf e_2}\ast {\bf e_3})&= {\bf e_2}\ast D({\bf e_3}) +{\bf e_3}\ast D({\bf e_2}) \\
\nonumber  &= {\bf e_2}\ast (b{\bf e_1}-b{\bf e_3})+{\bf e_3}\ast (a{\bf e_1}-a{\bf e_2}) \\
\label{equation29.4}  &=\begin{aligned}[t]
   b(\frac{1}{2}{\bf e_1}+\frac{1}{2}{\bf e_2})-b(p_{23,2}{\bf e_2}+p_{23,3}{\bf e_3})+a(\frac{1}{2}{\bf e_1}+\frac{1}{2}{\bf e_3})  \\
            -a (p_{23,2}{\bf e_2}+p_{23,3}{\bf e_3}).
        \end{aligned}
\end{align}By comparing coefficients of ${\bf e_1}$, ${\bf e_2}$, we get 
\begin{align}
\label{equation29.5} (a-b)(\frac{1}{2}-p_{23,2})=0, \\
\label{equation29.6}  b(\frac{1}{2}-p_{23,2}) =0. 
\end{align}Since $p_{23,2}\neq 1/2$, we get $a=b=0$ by (\ref{equation29.5}), (\ref{equation29.6}). We do a similar computation on $D({\bf e_2}\ast {\bf e_4})$ to get $c=0$. Hence, \[D({\bf e_2})=D({\bf e_3})=D({\bf e_4})=0.\]Next, we compute $D({\bf e_1}\ast {\bf e_2})$ in two ways. First,
\begin{align}
\nonumber D({\bf e_1}\ast {\bf e_2}) &= \frac{1}{2} D({\bf e_1}) +\frac{1}{2}D({\bf e_2})= \frac{1}{2}D({\bf e_1}) \\
\label{equation29.7}  &=\frac{1}{2} (d {\bf e_1} +e {\bf e_2} +f {\bf e_3} -(d+e+f) {\bf e_4}).
\end{align}Alternatively, 
\begin{align}
\nonumber D({\bf e_1}\ast {\bf e_2}) &= {\bf e_1}\ast D({\bf e_2}) +{\bf e_2}\ast D({\bf e_1})= {\bf e_1}\ast D({\bf e_1})\\
\nonumber &= {\bf e_2}\ast (d {\bf e_1}+e {\bf e_2}+f {\bf e_3} -(d+e+f) {\bf e_4}) \\
\label{equation29.8} &=\begin{aligned}[t]
   d(\frac{1}{2}{\bf e_1}+\frac{1}{2}{\bf e_2})+e{\bf e_2}+f(p_{23,2}{\bf e_2}+p_{23,3} {\bf e_3})  \\
           -(d+e+f)(p_{24,2}{\bf e_2} +p_{24,4}{\bf e_4}).
        \end{aligned}
\end{align}By comparing coefficients of ${\bf e_3}$, ${\bf e_4}$ in (\ref{equation29.7}), (\ref{equation29.8}), 
\begin{align}
\label{equation29.9}  f(\frac{1}{2}-p_{23,3}) =0,  \\
\label{equation29.91}  (d+e)(\frac{1}{2} -p_{24,4}) =0. 
\end{align}Since $p_{23,3}\neq 1/2$, $p_{24,4}\neq 1/2$, we get $f=0$, $d+e=0$ in (\ref{equation29.9}), (\ref{equation29.91}). Hence, \[D({\bf e_1})=d{\bf e_1}-d{\bf e_2}.\]We compute $D({\bf e_1}\ast {\bf e_3})$ in two ways. First,
\begin{align}
\label{equation29.92} D({\bf e_1}\ast {\bf e_3}) &=\frac{1}{2}D({\bf e_1})+\frac{1}{2}D({\bf e_3})= \frac{1}{2}D({\bf e_1})=\frac{1}{2}(d{\bf e_1}-d{\bf e_2}).
\end{align}Alternatively,
\begin{align}
\nonumber D({\bf e_1}\ast {\bf e_3}) &= {\bf e_3}\ast D({\bf e_1})={\bf e_3}\ast (d{\bf e_1}-d{\bf e_2})\\
\label{equation29.93} &= d(\frac{1}{2}{\bf e_1}+\frac{1}{2}{\bf e_3})-d(p_{23,2}{\bf e_2}+p_{23,3}{\bf e_3}).
\end{align}By comparing coefficients of ${\bf e_2}$ in (\ref{equation29.92}), (\ref{equation29.93}), we get 
\begin{align}
\label{equation29.94} d(\frac{1}{2}-p_{23,2}) =0. 
\end{align}Since $p_{23,2}\neq 1/2$, we get $d=0$. We summarize our results:

\begin{theorem} \label{caseE}
Let $(\mathbb{R}^4,\ast)$ be a four dimensional genetic Volterra algebra. In Case E: $|[4]_1|=3, |[4]_2|=1, |[4]_3|=1, |[4]_4|=1$, all derivations are trivial.
\end{theorem}

\subsection{Case F} \label{subsection6} By Lemma \ref{lemma1} and (\ref{equation2}), we have the following results: \[d_{14}=d_{41}=d_{23}=d_{32}=d_{24}=d_{42}=0,\]\[{\bf e_1}\ast {\bf e_2}=\frac{1}{2} {\bf e_1}+\frac{1}{2} {\bf e_2},\quad{\bf e_1}\ast {\bf e_3}=\frac{1}{2}{\bf e_1}+\frac{1}{2}{\bf e_3},\quad{\bf e_3}\ast {\bf e_4}=\frac{1}{2}{\bf e_3}+\frac{1}{2}{\bf e_4}.\]Let $a=_{21}$, $b=d_{34}$, $c=d_{11}$, $d=d_{12}$, $f=d_{31}$, $g=d_{33}$. By Lemma \ref{lemma2.5},
\[
D=(d_{ij})_{i,j\in [4]}=\left( {\begin{array}{cccc}
c & d & -(c+d) & 0\\
a & -a & 0 & 0\\
f & 0 & g & -(f+g) \\
0& 0 & b & -b\\
\end{array}}\right).
\]For $D({\bf e_1}\ast {\bf e_2})$, we compute it in two ways. First, 
\begin{align}
\nonumber D({\bf e_1}\ast {\bf e_2}) &= \frac{1}{2} D({\bf e_1})+\frac{1}{2} D({\bf e_2}) \\
\nonumber &= \frac{1}{2}(c{\bf e_1}+d{\bf e_2}-(c+d){\bf e_3})+\frac{1}{2} (a{\bf e_1}-a{\bf e_2}) \\
\label{equation30}  &= \frac{1}{2}(c+a){\bf e_1} +\frac{1}{2}(d-a){\bf e_2}-\frac{1}{2}(c+d) {\bf e_3}.
\end{align}Alternatively,
\begin{align}
\nonumber D({\bf e_1}\ast {\bf e_2}) &={\bf e_1}\ast D({\bf e_2})+{\bf e_2}\ast D({\bf e_1}) \\
\nonumber &= {\bf e_1}\ast (a{\bf e_1}-a{\bf e_2})+{\bf e_2}\ast (c{\bf e_1}+d{\bf e_2}-(c+d){\bf e_3})\\
\label{equation31} &= \begin{aligned}[t]
            a{\bf e_1}-a(\frac{1}{2}{\bf e_1}+\frac{1}{2} {\bf e_2})+c(\frac{1}{2}{\bf e_1}+\frac{1}{2}{\bf e_2})+d{\bf e_2} \\
            -(c+d) (p_{23,2}{\bf e_2}+p_{23,3}{\bf e_3}).
        \end{aligned}
\end{align}By comparing coefficients of ${\bf e_2}$ and ${\bf e_3}$ on both sides of (\ref{equation30}) and (\ref{equation31}), we get 
\begin{align}
\label{equation32}  (c+d)(\frac{1}{2} -p_{23,2}) &= 0 \\
\label{equation33}  (c+d)(\frac{1}{2} -p_{23,3}) &=0.
\end{align}Since $p_{23,2}\neq 1/2$ and $p_{23,3}\neq 1/2$, we get $c+d=0$ by either (\ref{equation31}) or (\ref{equation32}). We do a similar computation on $D({\bf e_3}\ast {\bf e_4})$ to get $f=0$. The derivation $D$ is in the following form: 
\[
D=(d_{ij})_{i,j\in [4]}=\left( {\begin{array}{cccc}
c & -c & 0 & 0\\
a & -a & 0 & 0\\
0 & 0 & g & -g \\
0& 0 & b & -b\\
\end{array}}\right).
\]We compute $D({\bf e_1}\ast {\bf e_3})$ in two ways. First, 
\begin{align}
\nonumber  D({\bf e_1}\ast {\bf e_3}) &= \frac{1}{2} D({\bf e_1})+\frac{1}{2} D({\bf e_3}) \\
\label{equation34}  &= \frac{1}{2}(c{\bf e_1}-c {\bf e_2})+\frac{1}{2} (g {\bf e_3} -g {\bf e_4}).
\end{align}Alternatively,
\begin{align}
\nonumber D({\bf e_1}\ast {\bf e_3}) &= {\bf e_1}\ast D({\bf e_3})+{\bf e_3}\ast D({\bf e_1}) \\
\nonumber &= {\bf e_1}\ast (g {\bf e_3} -g {\bf e_4})+{\bf e_3} \ast (c {\bf e_1} -c {\bf e_2})\\
\label{equation35}  &= \begin{aligned}[t]
            g(\frac{1}{2}{\bf e_1}+\frac{1}{2}{\bf e_3})-g(p_{14,1}{\bf e_1}+p_{14,4} {\bf e_4}) +c(\frac{1}{2}{\bf e_1}+\frac{1}{2} {\bf e_3})\\
            -c(p_{23,2} {\bf e_2}+p_{23,3} {\bf e_3}).
        \end{aligned}
\end{align}By comparing coefficients of ${\bf e_1}$ and ${\bf e_2}$ on both sides of (\ref{equation34}) and (\ref{equation35}),
\begin{align}
\label{equation36} g(\frac{1}{2} -p_{14,1}) &= 0, \\
\label{equation37} c(\frac{1}{2} -p_{23,2}) &=0.
\end{align}Since $p_{14,1}\neq 1/2$ and $p_{23,2}\neq 1/2$, we get $g=c=0$. Hence, $D({\bf e_1})=0$ and $D({\bf e_3})=0$. For $D({\bf e_2}\ast {\bf e_3})$, 
\begin{align}
\nonumber D({\bf e_2}\ast {\bf e_3}) &= {\bf e_2}\ast D({\bf e_3}) + {\bf e_3}\ast D({\bf e_2}) \\
\nonumber D(p_{23,2} {\bf e_2}+p_{23,3} {\bf e_3}) &= {\bf e_3} \ast D({\bf e_2})  \\
\nonumber p_{23,2} D({\bf e_2}) &= {\bf e_3} \ast (a {\bf e_1} -a {\bf e_2}) \\
\label{equation38} p_{23,2} (a {\bf e_1} -a {\bf e_2}) &= a(\frac{1}{2} {\bf e_1} +\frac{1}{2} {\bf e_3})-a (p_{23,2}{\bf e_2}+p_{23,3} {\bf e_3}).
\end{align}By comparing coefficients of ${\bf e_1}$ in (\ref{equation38}), we get 
\begin{align}
\label{equation39} a(\frac{1}{2} - p_{23,2}) &=0. 
\end{align}Since $p_{23,2}\neq 1/2$, we get $a=0$ in (\ref{equation39}). By using the fact that $D({\bf e_1})=0$, we do a similar computation on $D({\bf e_1}\ast {\bf e_4})$ to get $b=0$. We summarize our results:

\begin{theorem} \label{caseF}
Let $(\mathbb{R}^4,\ast)$ be a four dimensional genetic Volterra algebra. In Case F: $|[4]_1|=2, |[4]_2|=1, |[4]_3|=2, |[4]_4|=1$, all derivations are trivial.
\end{theorem}

\subsection{Case G} \label{subsection7} By Lemma \ref{lemma1} and (\ref{equation2}), we have the following results: \[d_{23}=d_{32}=d_{24}=d_{42}=0.\]\[{\bf e_1}\ast {\bf e_2}=\frac{1}{2}{\bf e_1}+\frac{1}{2}{\bf e_2},\quad {\bf e_1}\ast {\bf e_3}=\frac{1}{2} {\bf e_1}+\frac{1}{2}{\bf e_3},\]\[{\bf e_1}\ast {\bf e_4}=\frac{1}{2}{\bf e_1}+\frac{1}{2}{\bf e_4},\quad {\bf e_3}\ast {\bf e_4}=\frac{1}{2}{\bf e_3}+\frac{1}{2}{\bf e_4}.\]Let $a=d_{21}$, $b=d_{31}$, $c=d_{33}$, $d=d_{41}$, $e=d_{43}$, $f=d_{11}$, $g=d_{12}$, $h=d_{13}$. By Lemma \ref{lemma2.5},
\[
D=(d_{ij})_{i,j\in [4]}=\left( {\begin{array}{cccc}
f & g & h & -(f+g+h)\\
a & -a & 0 & 0\\
b & 0 & c & -(b+c) \\
d& 0 & e & -(d+e)\\
\end{array}}\right).
\]We compute $D({\bf e_2}\ast {\bf e_3})$ in two ways. First, 
\begin{align}
\nonumber D({\bf e_2}\ast {\bf e_3}) &= p_{23,2}D({\bf e_2})+p_{23,3}D({\bf e_3}) \\
\label{equation39.4}  &= p_{23,2}(a{\bf e_1}-a{\bf e_2})+p_{23,3}(b{\bf e_1}+c{\bf e_3}-(b+c){\bf e_4}).
\end{align}Alternatively, 
\begin{align}
\nonumber D({\bf e_2}\ast {\bf e_3}) &= {\bf e_2}\ast D({\bf e_3})+ {\bf e_3}\ast D({\bf e_2}) \\
\nonumber &= {\bf e_2}\ast (b{\bf e_1}+c{\bf e_3}-(b+c){\bf e_4})+{\bf e_3}\ast (a{\bf e_1}-a{\bf e_2}) \\
\label{equation39.5} &=\begin{aligned}[t]
           b(\frac{1}{2}{\bf e_1}+\frac{1}{2}{\bf e_2})+c(p_{23,2}{\bf e_2}+p_{23,3}{\bf e_3})-(b+c)(p_{24,2}{\bf e_2}\\
            +p_{24,4}{\bf e_4})+a(\frac{1}{2}{\bf e_1}+\frac{1}{2}{\bf e_3})-a(p_{23,2}{\bf e_2}+p_{23,3}{\bf e_3}).
        \end{aligned}
\end{align}By comparing coefficients of ${\bf e_1}$, ${\bf e_3}$ on both sides of (\ref{equation39.4}) and (\ref{equation39.5}), 
\begin{align}
\label{equation39.6} (a-b)(\frac{1}{2}-p_{23,2})=0,  \\
\label{equation39.7} a(\frac{1}{2}-p_{23,3})=0. 
\end{align}Since $p_{23,2}\neq 1/2$, $p_{23,3}\neq 1/2$, we get $a=b=0$ by (\ref{equation39.6}), (\ref{equation39.7}). So, $D({\bf e_2})=0$. By comparing coefficients of ${\bf e_2}$, ${\bf e_4}$ on both sides of (\ref{equation39.4}) and (\ref{equation39.5}), 
\begin{align}
\label{equation39.8} c(p_{23,2}-p_{24,2}) &=0, \\
\label{equation39.9} c(p_{23,3}-p_{24,4}) &=0. 
\end{align}We do a similar computation on $D({\bf e_2}\ast {\bf e_4})$ to get $d=0$ and 
\begin{align}
\label{equation39.91} e(p_{23,2} -p_{24,2}) &= 0,\\
\label{equation39.92} e(p_{23,3} -p_{24,4}) &= 0. 
\end{align}By (\ref{equation39.8}), (\ref{equation39.9}), (\ref{equation39.91}), (\ref{equation39.92}), if $p_{23,2}=p_{24,2}$ (it automatically implies that $p_{23,3}=p_{24,4}$), then $c=e=0$. Next, we compute $D({\bf e_1}\ast {\bf e_2})$ in two ways. First, 
\begin{align}
\nonumber D({\bf e_1}\ast {\bf e_2}) &= D (\frac{1}{2}{\bf e_1}+\frac{1}{2}{\bf e_2})=\frac{1}{2}D({\bf e_1})\\
\label{equation39.93}   &=\frac{1}{2} (f{\bf e_1} +g{\bf e_2} +h {\bf e_3}-(f+g+h) {\bf e_4}).
\end{align}Alternatively,
\begin{align}
\nonumber D({\bf e_1}\ast {\bf e_2}) &= {\bf e_1}\ast D({\bf e_2}) + {\bf e_2} \ast D({\bf e_1})= {\bf e_2}\ast D({\bf e_1}) \\
\nonumber &= {\bf e_2}\ast (f{\bf e_1}+g {\bf e_2}+h {\bf e_3}-(f+g+h) {\bf e_4})\\
\label{equation39.94}  &=\begin{aligned}[t]
           f(\frac{1}{2}{\bf e_1}+\frac{1}{2}{\bf e_2})+g{\bf e_2}+h(p_{23,2}{\bf e_2}+p_{23,3}{\bf e_3})\\
            -(f+g+h) (p_{24,2}{\bf e_2} +p_{24,4} {\bf e_4}).
        \end{aligned}
\end{align}By comparing coefficients of ${\bf e_3}$, ${\bf e_4}$ on both sides of (\ref{equation39.93}) and (\ref{equation39.94}), 
\begin{align}
\label{equation39.95} h(\frac{1}{2}-p_{23,3}) &=0, \\
\label{equation39.96} (f+g) (\frac{1}{2} -p_{24,4}) &=0.
\end{align}Since $p_{23,3}\neq 1/2$, $p_{24,4}\neq 1/2$, $h=0$ and $f+g=0$ by (\ref{equation39.95}), (\ref{equation39.96}). We update the form of the derivation $D$ as follows:
\[
D=(d_{ij})_{i,j\in [4]}=\left( {\begin{array}{cccc}
f & -f & 0 & 0\\
0 & 0 & 0 & 0\\
0 & 0 & c & -c \\
0& 0 & e & -e\\
\end{array}}\right).
\]We compute $D({\bf e_1}\ast {\bf e_3})$ in two ways. First,
\begin{align}
\nonumber D({\bf e_1}\ast {\bf e_3})&=\frac{1}{2}D({\bf e_1}) +\frac{1}{2}D({\bf e_3}) \\
\label{equation39.97}  &= \frac{1}{2}(f{\bf e_1}-f{\bf e_2})+\frac{1}{2} (c{\bf e_3}-c{\bf e_4}).
\end{align}Alternatively, 
\begin{align}
\nonumber D({\bf e_1}\ast {\bf e_3}) &= {\bf e_1}\ast D({\bf e_3}) + {\bf e_3}\ast D({\bf e_1}) \\
\nonumber &= {\bf e_1}\ast (c{\bf e_3}-c{\bf e_4})+{\bf e_3}\ast (f{\bf e_1}-f{\bf e_2}) \\
\label{equation39.98} &=\begin{aligned}[t]
           c(\frac{1}{2}{\bf e_1}+\frac{1}{2}{\bf e_3})-c(\frac{1}{2}{\bf e_1}+\frac{1}{2}{\bf e_4})+f(\frac{1}{2}{\bf e_1}+\frac{1}{2}{\bf e_3})\\
            -f(p_{23,2}{\bf e_2}+p_{23,3}{\bf e_3}).
        \end{aligned}
\end{align}By comparing coefficients of ${\bf e_3}$ on both sides of (\ref{equation39.97}), (\ref{equation39.98}),
\begin{align}
\label{equation39.99} f(\frac{1}{2}-p_{23,2}) =0. 
\end{align}Since $p_{23,2}\neq 1/2$, we get $f=0$ in (\ref{equation39.99}). So, the derivation $D$ is essentially the same as the derivations appeared in Case A. We summarize our results:
\begin{theorem} \label{caseG}
Let $(\mathbb{R}^4,\ast)$ be a four dimensional genetic Volterra algebra. In Case G: $|[4]_1|=3, |[4]_2|=1, |[4]_3|=2, |[4]_4|=2$, derivations are in the form 
\[
D=(d_{ij})_{i,j\in [4]}=\left( {\begin{array}{cccc}
0 & 0 & 0 & 0\\
0 & 0 & 0 & 0\\
0 & 0 & c & -c \\
0& 0 & e & -e\\
\end{array}}\right).
\]If $D$ is non-trivial (one of $c$ and $e$ is non-zero), then $p_{23,2}=p_{24,2}$ and $p_{23,3}=p_{24,4}$. 
\end{theorem}

\subsection{Case H} \label{subsection8} By Lemma \ref{lemma1} and (\ref{equation2}), we have the following results: \[d_{13}=d_{31}=d_{24}=d_{42}=0.\]\[{\bf e_1}\ast {\bf e_2}=\frac{1}{2}{\bf e_1}+\frac{1}{2}{\bf e_2},\quad {\bf e_1}\ast {\bf e_4}=\frac{1}{2} {\bf e_1}+\frac{1}{2}{\bf e_4},\]\[{\bf e_2}\ast {\bf e_3}=\frac{1}{2}{\bf e_2}+\frac{1}{2}{\bf e_3},\quad {\bf e_3}\ast {\bf e_4}=\frac{1}{2}{\bf e_3}+\frac{1}{2}{\bf e_4}.\]Let $a=d_{11}$, $b=d_{12}$, $c=d_{21}$, $d=d_{22}$, $e=d_{32}$, $f=d_{33}$, $g=d_{41}$, $h=d_{43}$. By Lemma \ref{lemma2.5}, 
\[
D=(d_{ij})_{i,j\in [4]}=\left( {\begin{array}{cccc}
a & b & 0 & -(a+b)\\
c & d & -(c+d) & 0\\
0 & e & f & -(e+f) \\
g & 0 & h & -(g+h)\\
\end{array}}\right).
\]We compute $D({\bf e_1}\ast {\bf e_2})$ in two ways. First,
\begin{align}
\nonumber D({\bf e_1}\ast {\bf e_2}) &= \frac{1}{2} D({\bf e_1})+\frac{1}{2} D({\bf e_2}) \\
\label{equation41}  &=\frac{1}{2}(a+c){\bf e_1}+\frac{1}{2}(b+d){\bf e_2}-\frac{1}{2}(c+d){\bf e_3}-\frac{1}{2}(a+b) d_{14}{\bf e_4}.
\end{align}Alternatively,
\begin{align}
\nonumber D({\bf e_1}\ast {\bf e_2}) &= {\bf e_1}\ast D({\bf e_2}) + {\bf e_2}\ast D({\bf e_1}) \\
\nonumber &={\bf e_1}\ast (c {\bf e_1}+d {\bf e_2}-(c+d) {\bf e_3})+{\bf e_2}\ast (a{\bf e_1}+b{\bf e_2}-(a+b){\bf e_4})\\
\label{equation42} &=\begin{aligned}[t]
           c{\bf e_1}+d(\frac{1}{2}{\bf e_1}+\frac{1}{2}{\bf e_2})-(c+d)(p_{13,1}{\bf e_1}+p_{13,3}{\bf e_3})\\
            a (\frac{1}{2}{\bf e_1}+\frac{1}{2}{\bf e_2})+b{\bf e_2}-(a+b)(p_{24,2}{\bf e_2}+p_{24,4} {\bf e_4}).
        \end{aligned}
\end{align}By comparing coefficients of ${\bf e_1}$ in (\ref{equation41}) and (\ref{equation42}), we get 
\begin{align}
\nonumber (c+d)(\frac{1}{2} -p_{13,1}) &= 0.
\end{align}Since $p_{13,1}\neq 1/2$, we get $c+d=0$. By comparing coefficients of ${\bf e_2}$ in (\ref{equation41}) and (\ref{equation42}), we get 
\begin{align}
\nonumber (a+b)(\frac{1}{2}-p_{24,2})=0.
\end{align}Since $p_{24,2}\neq 1/2$, we get $a+b=0$. We do similar computations on $D({\bf e_3}\ast {\bf e_4})$ to get $e=0$ and $g=0$. Then, the derivation $D$ is in the form:
\[
D=(d_{ij})_{i,j\in [4]}=\left( {\begin{array}{cccc}
a & -a & 0 & 0\\
c & -c & 0  & 0\\
0 & 0 & f & -f \\
0 & 0 & h & -h\\
\end{array}}\right).
\]We compute $D({\bf e_1}\ast {\bf e_3})$ in two ways. First, 
\begin{align}
\nonumber D({\bf e_1}\ast {\bf e_3}) &= p_{13,1}D({\bf e_1}) +p_{13,3} D({\bf e_3})\\
\label{equation43} &= p_{13,1}(a{\bf e_1}-a {\bf e_2})+p_{13,3} (f {\bf e_3}-f {\bf e_4}).
\end{align}Alternatively, 
\begin{align}
\nonumber D({\bf e_1}\ast {\bf e_3}) &= {\bf e_1}\ast D({\bf e_3}) +{\bf e_3} \ast D({\bf e_1}) \\
\nonumber &= {\bf e_1}\ast (f{\bf e_3} -f {\bf e_4})+{\bf e_3}\ast (a{\bf e_1}-a {\bf e_2}) \\
\label{equation44} &=\begin{aligned}[t]
           f(p_{13,1}{\bf e_1}+p_{13,3}{\bf e_3})-f (\frac{1}{2}{\bf e_1}+\frac{1}{2} {\bf e_4})+a (p_{13,1} {\bf e_1}+p_{13,3}{\bf e_3}  )  \\
            -a(\frac{1}{2}{\bf e_2}+\frac{1}{2}{\bf e_3}).
        \end{aligned}
\end{align}By comparing coefficients of ${\bf e_1}$ and ${\bf e_2}$ in (\ref{equation43}) and (\ref{equation44}), we get 
\begin{align}
\label{equation45}  f(\frac{1}{2}-p_{13,1})=0 \\ 
\label{equation46}  a(\frac{1}{2}-p_{13.1})=0.
\end{align}Since $p_{13,1}\neq 1/2$, we get $a=f=0$ by (\ref{equation45}) and (\ref{equation46}). We do similar computations on $D({\bf e_2}\ast {\bf e_4})$ to get $c=h=0$. We summarize our results:

\begin{theorem} \label{caseH}
Let $(\mathbb{R}^4,\ast)$ be a four dimensional genetic Volterra algebra. In Case H: $|[4]_1|=2, |[4]_2|=2, |[4]_3|=2, |[4]_4|=2$, all derivations are trivial.
\end{theorem}

\subsection{Case I} \label{subsection9} By Lemma \ref{lemma1} and (\ref{equation2}), we have the following results: \[d_{24}=d_{42}=0.\]\[{\bf e_1}\ast {\bf e_2}=\frac{1}{2}{\bf e_1}+\frac{1}{2}{\bf e_2},\quad {\bf e_1}\ast {\bf e_3}=\frac{1}{2}{\bf e_1}+\frac{1}{2}{\bf e_3}, \quad {\bf e_1}\ast {\bf e_4}=\frac{1}{2}{\bf e_1}+\frac{1}{2}{\bf e_4},\]\[{\bf e_2}\ast {\bf e_3}=\frac{1}{2}{\bf e_2}+\frac{1}{2}{\bf e_3}, \quad {\bf e_3}\ast {\bf e_4} =\frac{1}{2}{\bf e_3}+\frac{1}{2}{\bf e_4}.\]Let $a=d_{11}$, $b=d_{12}$, $c=d_{13}$, $d=d_{21}$, $e=d_{22}$, $f=d_{31}$, $g=d_{32}$, $h=d_{33}$, $i=d_{41}$ and $j=d_{43}$. By Lemma \ref{lemma2.5},
\[
D=(d_{ij})_{i,j\in [4]}=\left( {\begin{array}{cccc}
a & b & c & -(a+b+c)\\
d & e & -(d+e)  & 0\\
f & g & h & -(f+g+h) \\
i & 0 & j & -(i+j)\\
\end{array}}\right).
\]We compute $D({\bf e_2}\ast {\bf e_4})$ in two ways. First, 
\begin{align}
\nonumber D({\bf e_2}\ast {\bf e_4}) &= p_{24,2} D({\bf e_2})+p_{24,4} D({\bf e_4})\\
\label{equation49} &=p_{24,2}(d{\bf e_1}+e{\bf e_2} -(d+e){\bf e_3})+p_{24,4} (i{\bf e_1}+j {\bf e_3}-(i+j){\bf e_4}).
\end{align}Alternatively, 
\begin{align}
\nonumber D({\bf e_2}\ast {\bf e_4}) &= {\bf e_2}\ast D({\bf e_4}) + {\bf e_4}\ast D({\bf e_2}) \\
\nonumber &={\bf e_2}\ast (i{\bf e_1}+j {\bf e_3}-(i+j){\bf e_4})+{\bf e_4}\ast (d{\bf e_1}+e {\bf e_2}-(d+e){\bf e_3})\\
\label{equation50} &= \begin{aligned}[t]
             i(\frac{1}{2}{\bf e_1}+\frac{1}{2}{\bf e_2})+j (\frac{1}{2}{\bf e_2}+\frac{1}{2}{\bf e_3})-(i+j)(p_{24,2}{\bf e_2}+p_{24,4}{\bf e_4}) \\
            +d(\frac{1}{2}{\bf e_1}+\frac{1}{2}{\bf e_4})+e(p_{24,2}{\bf e_2}+p_{24,4}{\bf e_4})-(d+e)(\frac{1}{2}{\bf e_3}+\frac{1}{2}{\bf e_4}).
        \end{aligned}
\end{align}By comparing coefficients of ${\bf e_1}$, ${\bf e_2}$, ${\bf e_3}$ on both sides of (\ref{equation49}) and (\ref{equation50}), we get 
\begin{align}
\label{equation51}  (d-i)(\frac{1}{2}-p_{24,2}) &=0 \\
\label{equation52}  (i+j)(\frac{1}{2}-p_{24,2}) &=0\\
\label{equation53}  (j+d+e)(\frac{1}{2}-p_{24,2}) &=0. 
\end{align} Since $p_{24,2}\neq 1/2$, we know that $d=i=-j$ and $e=0$ by (\ref{equation51}), (\ref{equation52}), (\ref{equation53}). The derivation $D$ is in the following form:
\[
D=(d_{ij})_{i,j\in [4]}=\left( {\begin{array}{cccc}
a & b & c & -(a+b+c)\\
d & 0 & -d  & 0\\
f & g & h & -(f+g+h) \\
d & 0 & -d & 0\\
\end{array}}\right).
\]We compute $D({\bf e_2}\ast {\bf e_3})$ in two ways. First, 
\begin{align}
\nonumber D({\bf e_2}\ast {\bf e_3}) &= \frac{1}{2} D({\bf e_2})+\frac{1}{2} D({\bf e_3}) \\
\label{equation54} &=\frac{1}{2} (d{\bf e_1}-d {\bf e_3})+\frac{1}{2} (f{\bf e_1}+g{\bf e_2}+h{\bf e_3}-(f+g+h){\bf e_4}).
\end{align}Alternatively, 
\begin{align}
\nonumber D({\bf e_2}\ast {\bf e_3}) &= {\bf e_2}\ast D({\bf e_3}) +{\bf e_3}\ast D({\bf e_2}) \\
\nonumber &={\bf e_2}\ast (f{\bf e_1}+g{\bf e_2}+h{\bf e_3}-(f+g+h){\bf e_4})+{\bf e_3}\ast (d{\bf e_1}-d{\bf e_3})\\
\label{equation55} &=\begin{aligned}[t]
             f(\frac{1}{2}{\bf e_1}+\frac{1}{2}{\bf e_2})+g{\bf e_2}+h(\frac{1}{2}{\bf e_2}+\frac{1}{2}{\bf e_3}) -(f+g+h)(p_{24,2}{\bf e_2}\\
             +p_{24,4}{\bf e_4})+d(\frac{1}{2}{\bf e_1}+\frac{1}{2}{\bf e_3})-d{\bf e_3}.
        \end{aligned}
\end{align}By comparing coefficients of ${\bf e_2}$ on both sides of (\ref{equation54}) and (\ref{equation55}), we get 
\begin{align}
\label{equation56} (f+g+h)(\frac{1}{2}-p_{24,2}) =0.
\end{align}Since $p_{24,2}\neq 1/2$, we get $f+g+h=0$ by (\ref{equation56}). We do a similar computation on $D({\bf e_1}\ast {\bf e_2})$ to get $a+b+c=0$. We update the form of the derivation $D$ as follows:
\[
D=(d_{ij})_{i,j\in [4]}=\left( {\begin{array}{cccc}
a & b & -(a+b) & 0\\
d & 0 & -d  & 0\\
f & g & -(f+g) & 0 \\
d & 0 & -d & 0\\
\end{array}}\right).
\]We compute $D({\bf e_1}\ast {\bf e_4})$ in two ways. First, 
\begin{align}
\nonumber  D({\bf e_1}\ast {\bf e_4})&= \frac{1}{2} D({\bf e_1})+\frac{1}{2} D({\bf e_4}) \\
\label{equation57}  &= \frac{1}{2}(a{\bf e_1}+b{\bf e_2}-(a+b){\bf e_3})+\frac{1}{2}(d{\bf e_1}-d{\bf e_3}).
\end{align}Alternatively,
\begin{align}
\nonumber  D({\bf e_1}\ast {\bf e_4}) &= {\bf e_1}\ast D({\bf e_4}) + {\bf e_4}\ast D({\bf e_1})\\
\nonumber &= {\bf e_1}\ast (d{\bf e_1}-d{\bf e_3})+{\bf e_4}\ast (a{\bf e_1}+b{\bf e_2}-(a+b){\bf e_3})\\
\label{equation58} &=\begin{aligned}[t]
             d{\bf e_1}-d(\frac{1}{2}{\bf e_1}+\frac{1}{2}{\bf e_3})+a(\frac{1}{2}{\bf e_1}+\frac{1}{2}{\bf e_4})+b(p_{24,2}{\bf e_2}+p_{24,4}{\bf e_4})\\
             -(a+b)(\frac{1}{2}{\bf e_3}+\frac{1}{2}{\bf e_4}).
        \end{aligned}
\end{align}By comparing coefficients of ${\bf e_2}$ on both sides of (\ref{equation57}) and (\ref{equation58}), we get 
\begin{align}
\label{equation59} b(\frac{1}{2}-p_{24,2}) &=0.
\end{align}Since $p_{24,2} \neq 1/2$, we know that $b=0$ by (\ref{equation59}). We do a similar computation on $D({\bf e_3}\ast {\bf e_4})$ to get $g=0$. We summarize our results:

\begin{theorem} \label{caseI}
Let $(\mathbb{R}^4,\ast)$ be a four dimensional genetic Volterra algebra. In Case I: $|[4]_1|=3, |[4]_2|=2, |[4]_3|=3, |[4]_4|=2$, derivations are in the form:
\[
D=(d_{ij})_{i,j\in [4]}=\left( {\begin{array}{cccc}
a & 0 & -a & 0\\
d & 0 & -d  & 0\\
f & 0 & -f & 0 \\
d & 0 & -d & 0\\
\end{array}}\right).
\]
\end{theorem}

\subsection{Case J} \label{subsectionJ} By (\ref{equation2}), we have \[{\bf e_1}\ast {\bf e_2}=\frac{1}{2}{\bf e_1}+\frac{1}{2}{\bf e_2},\quad {\bf e_1}\ast {\bf e_3}=\frac{1}{2}{\bf e_1}+\frac{1}{2}{\bf e_3}, \quad {\bf e_1}\ast {\bf e_4}=\frac{1}{2}{\bf e_1}+\frac{1}{2}{\bf e_4},\]\[{\bf e_2}\ast {\bf e_3}=\frac{1}{2}{\bf e_2}+\frac{1}{2}{\bf e_3},\quad {\bf e_2}\ast {\bf e_4}=\frac{1}{2}{\bf e_2}+\frac{1}{2}{\bf e_4}, \quad {\bf e_3}\ast {\bf e_4} =\frac{1}{2}{\bf e_3}+\frac{1}{2}{\bf e_4}.\]By Lemma \ref{lemma2.5}, we summarize our results as follows: 
\begin{theorem} \label{caseJ}
Let $(\mathbb{R}^4,\ast)$ be a four dimensional genetic Volterra algebra. In Case J: $|[4]_1|=3, |[4]_2|=3, |[4]_3|=3, |[4]_4|=3$, the sum of each row in any derivation $D=(d_{ij})_{i,j\in [4]}$ is zero.  
\end{theorem}

\section{An application}  \label{section3}

Let $A$ be an algebra. A linear map $\bigtriangleup\colon A\rightarrow A$ is called {\it a local derivation} if, for any $x\in A$, there exists a derivation $D_{x}\colon A\rightarrow A$ such that $\bigtriangleup (x)=D_{x} (x)$.

\begin{theorem}
Let $(\mathbb{R}^4, \ast)$ be a four dimensional genetic Volterra algebra. Then any local derivation of $(\mathbb{R}^4,\ast)$ is a derivation.
\end{theorem}

\begin{proof}
By Theorem \ref{caseA} to Theorem \ref{caseJ}, non-trivial derivations exist in the following cases (up to a permutation of basis elements in $\mathbb{R}^4$):
 \renewcommand{\labelenumi}{\Roman{enumi}}
\begin{enumerate}
  \item $p_{12,1}=1/2$, $p_{13,1}=p_{23,2}\neq 1/2$, $p_{14,1}=p_{24,2}\neq 1/2, p_{34,3}\neq 1/2$.
  \item $p_{12,1}=p_{34,3}=1/2$, either $\{p_{13,1}=p_{23,2}\neq 1/2, p_{14,1}=p_{24,2}\neq 1/2\}$ or $\{p_{13,1}=p_{14,1}\neq 1/2, p_{23,2}=p_{24,2}\neq 1/2\}$. 
  \item $p_{12,1}=p_{13,1}=p_{23,2}=1/2$, $p_{14,1}\neq 1/2, p_{24,2}\neq 1/2, p_{34,3}\neq 1/2$.
  \item $p_{12,1}=p_{13,1}=p_{14,1}=p_{34,3}=1/2$, $p_{23,2}=p_{24,2}\neq 1/2$.
  \item $p_{12,1}=p_{13,1}=p_{14,1}=p_{23,2}=p_{34,3}=1/2$, $p_{24,2}\neq 1/2$.
  \item $p_{12,1}=p_{13,1}=p_{14,1}=p_{23,2}=p_{24,2}=p_{34,3}=1/2$.
\end{enumerate}In Case I, let $\bigtriangleup_I$ be a local derivation. By definition, there are derivations $D_{{\bf e_1}}$, $D_{{\bf e_2}}$, $D_{{\bf e_3}}$, $D_{{\bf e_4}}$ such that $\bigtriangleup_I({\bf e_1})=D_{{\bf e_1}}({\bf e_1})$, $\bigtriangleup_I({\bf e_2})=D_{{\bf e_2}}({\bf e_2})$, $\bigtriangleup_I({\bf e_3})=D_{{\bf e_3}}({\bf e_3})$, $\bigtriangleup_I({\bf e_4})=D_{{\bf e_4}}({\bf e_4})$. By Theorem \ref{caseA}, we have
\begin{align}
\nonumber D_{{\bf e_1}}({\bf e_1})=a({\bf e_1}-{\bf e_2}), D_{{\bf e_2}}({\bf e_2})=b({\bf e_1}-{\bf e_2}), D_{{\bf e_3}}({\bf e_3})=D_{{\bf e_4}}({\bf e_4})=0
\end{align}for some constants $a,b$. It means that $\bigtriangleup_I$ is defined by
\begin{align}
\nonumber \bigtriangleup_I({\bf e_1})=a({\bf e_1}-{\bf e_2}), \bigtriangleup_I({\bf e_2})=b({\bf e_1}-{\bf e_2}),  \bigtriangleup_I({\bf e_3})=\bigtriangleup_I({\bf e_4})=0.
\end{align}One easily verify that \[\bigtriangleup_I({\bf e_i}\ast {\bf e_j})={\bf e_i}\ast \bigtriangleup_I({\bf e_j})+{\bf e_j}\ast \bigtriangleup_I({\bf e_i})\]for all $i,j\in [4]$ and hence $\bigtriangleup_I$ is a derivation.
For the remaining cases, we deduce that any local derivation must be a derivation by similar reasoning as in Case I.
\end{proof}

\section{Acknowledgement}
The author is grateful to Farrukh Mukhamedov for some valuable discussions on the topic.

\Address

\end{document}